\documentclass[10pt]{article}
\usepackage{amsmath}
\usepackage{amssymb}
\usepackage{indentfirst}
\usepackage[latin1]{inputenc}
\usepackage{geometry}
\usepackage{amsfonts}
\usepackage{graphicx}
\usepackage{subfigure}
\usepackage{float}

\newcommand{\R}{{\mathbb R}}
\newcommand{\C}{{\mathbb C}}

\newcommand{\HQ}{{\mathbb H}}
\newcommand{\BQ}{{\mathbb B}}

\newcommand{\vp}{\vspace{0.3cm}}

\newcommand{\bc}{\begin{center}}
\newcommand{\ec}{\end{center}}

\newcommand{\Fe}{{\cal F}}

\newcommand{\SL}{{\rm SL}}
\newcommand{\PSL}{{\rm PSL}}
\newcommand{\SU}{{\rm SU}}

\newcommand{\Iso}{{\rm ISO}}
\newcommand{\tr}{{\rm tr}}

\newcommand{\qed}{\enspace\vrule  height6pt  width4pt  depth2pt}
\newenvironment{proof}{\par\noindent{\bf Proof.}}{$\qed$\par\bigskip}

\newtheorem{theorem}{Theorem}[section]

\newtheorem{lemma}[theorem]{Lemma}

\newtheorem{proposition}[theorem]{Proposition}

\begin{document}

\title{A Note on Discrete Groups\thanks{Mathematics subject Classification Primary
[$30F40$]; Secondary [$20H10$].
Keywords and phrases: Hyperbolic space, Isometric  sphere, Bisector, Canonical region.
\newline  Research partially supported by CNPq, UFPB and UNIVASF.
 }}

\author{ S. O. Juriaans, S.C. Lima Neto,  A. De A. E Silva}
\date{}

\maketitle

\begin{abstract}
We  prove that a Kleinian groups has a DF domain if and only if it has a DC domain . The Fuchsian case has recently been considered,  it was shown that, in this case, there are no cocompact examples and  cocompact Kleinian  examples were given. Here we prove that,  in the Kleinian case, there are no cocompact torsion free examples  and we describe the symmetries of a fundamental domain of such a group.
 \end{abstract}


\section{Introduction}

In \cite{kijusiso} it is proved that in hyperbolic 2 and 3-space the isometric spheres, in the ball models, are also the Poincaré bisectors. This was used to get explicit formulas for the Poincaré bisectors in hyperbolic 2 and 3-space. Using these formulas, generators were found for discrete groups of quaternions division algebras and  Poincaré fundamental polygons were constructed for the Bianchi groups and the Figure Eight Knot group.   Note that in \cite{johansson, katok, swan} similar questions are addressed.

  Two interesting problems are that of deciding when a Ford fundamental domain coincides with a Poincaré fundamental domain  (called a Dirichlet-Ford domain or DF domains) and when a Poincaré fundamental domain has more than one center (called a Double Dirichlet domain or DC domain). These problems were raised  in \cite{lakeland} and solved, in the same paper,  for Fuchsian groups. In particular it is proved that there are no cocompact examples in this case.  In \cite{kijusiso} an independent proof was given and an algebraic criterium was  established  which  the set of side-pairing transformations must satisfy. Actually it turns out that, in the Fuchsian case, these two problems have  identical solutions (\cite{lakeland}) and the question remained to see what happens in the Kleinian case.

Our main result in this paper is to solve above mentioned problems for Kleinian groups.  In particular, we show that also in this case, they are identical. Cocompact examples are constructed in \cite{lakeland} and here we show that no cocompact torsion free examples exist. The difference with the Fuchsian case lies in the possible number of linear orthogonal maps $A$ which arise as one writes a hyperbolic isometry $\gamma =A\sigma$, where $\ \sigma$ is the reflection in the isometric sphere. In the Fuchsian case only one such reflections shows up, namely the reflection in the imaginary axis. In the Kleinian case, we first give a rather good description of $A$ and use this to show that in a DF domain all of this linear maps, coming from the sideparing transformations, have a common eigenvector. If the group is torsion free then the direction of this eigenvector determines an ideal vertex. Together with the results of \cite{kijusiso} this gives an algebraic characterization, in term of a set of sidepairing transformations, of the Kleinian groups having a DF domain.   A part from from solving the above mentioned  problems for  Kleinian groups, we study the symmetry of their fundamental domain and derive some consequences of the fact that the isometric spheres, in the ball models of hyperbolic space, are the Poincaré bisectors in any dimension.  

The layout of the paper is as follows. In Section~\ref{seciso} we recall fundamentals and also some results of \cite{kijusiso} that we will need in the sequel. In Section~\ref{dfdomains} we settle the DF and DC problems for Kleinian groups. In Section~\ref{ford}, we study the symmetries of a Kleininan group having a DF domain and make some considerations on Kleinian groups and the bisectors of their elements. Most of the notation used is standard or follows that introduced in \cite{kijusiso}.


\section{Poincar\'e Bisectors}
\label{seciso}

In this section we recall  basic facts on hyperbolic spaces,  fix
notation and generalize a result of \cite{kijusiso}. Standard references are \cite{beardon, bridson, 
elstrodt, gromov, ratcliffe}. By  $\HQ^n$ (respectively $\BQ^n$) we denote the upper half space (plane) (respectively the ball) model of hyperbolic $n$-space.

The hyperbolic distance $\rho$ in $\HQ^3$  is determined by     $\cosh \rho(P,P')= \delta(P,P') = 1+\frac{d(P,P')^2}{2rr'},$ where $d$ is the Euclidean distance and $P=z+rj$,  $P'=z'+r'j$
are two elements of $\ \HQ^3$.

Let $\Gamma$ be a discrete subgroup of $\  {\rm Iso}^+(\BQ^3)$. The Poincar\'e method  can be used to give a presentation of $\Gamma$ (see for example \cite{ratcliffe}).  Let $\Gamma_0$ be the stabilizer in $\Gamma$ of $\ 0\in \BQ^{3}$ and let
$\Fe_0$ a fundamental domain for $\Gamma_0$.
 For $\gamma\in {\rm Iso}^+(\BQ^3)$,  let  $D_{\gamma}(0)=\{u\in \BQ\ | \ \rho
(0,u)\leq \rho (u,\gamma (0))\}$ and $\Sigma_{\gamma} = \{u\in \BQ\ | \ \rho
(0,u)= \rho (u,\gamma (0))\}$ (see \cite{kijusiso}) .
Then $\Fe=\Fe_0\cap
(\bigcap\limits_{\gamma\in
\Gamma \setminus \Gamma_0}D_{\gamma}(0))$  is a
Poincar\'{e} fundamental domain of $\ \Gamma$ with
center $0$.     If $\Gamma_0=1$ then $0$ is in the interior of the fundamental domain.

If $S_1$ and $S_2$ are two intersecting spheres
in the extended hyperbolic space, then
$(S_1,S_2)$ denotes the cosinus of the angle at
which they intersect, the dihedral angle.  This notation is taken from \cite{beardon}. Elements $x$ and $y$ of hyperbolic space  are inverse points with respect to $S_1$ if $y=\sigma (x)$, where $\sigma$ is the reflection in $S_1$. In case $S_1=\partial \BQ^3 = S^2$, the boundary of $\ \BQ^3$, then the inverse point of $\ x$ with respect to $S_1$ is denoted by $x^*$.

\vp

Let $\gamma \in \PSL (2,\C )$, $z_0\in \BQ^3$ and $\Psi :\PSL (2,\C )\rightarrow \  {\rm Iso}^+(\BQ^3)$ an isomorphism.  One can identify ${\rm Iso}^+(\BQ^3)$ with a subgroup of two by two matrices over the quaternions over the reals (see \cite{elstrodt}). Define $\Sigma_{\Psi (\gamma)}(z_0):=\{u\in \BQ^3\ |\ \rho (z_0,u)=\rho (u,\Psi (\gamma)^{-1}(z_0))\}$. Clearly $\Sigma_{\Psi (\gamma)}(0)=\Sigma_{\Psi (\gamma)}$ and $\Psi (\gamma_1)(\Sigma_{\Psi (\gamma_1^{-1}\gamma\gamma_1)})=\Sigma_{\Psi (\gamma)}(\gamma_1(0))$. If $\Gamma$ is a Kleinian group then define $D_{\Gamma}(z_0)$ as the intersection of $\BQ^3$ and the closure of $\bigcap\limits_{\gamma \in \Gamma}{\rm Exterior} (\Sigma_{\Psi (\gamma)}(z_0))$. We have that $D_{\Gamma}= D_{\Gamma}(0)$ and $D_{\Gamma}(\gamma_1(0))=\gamma_1(D_{\gamma_1^{-1}\Gamma\gamma_1})$.

For $\gamma =\begin{pmatrix}
a & b\\
c & d
\end{pmatrix}\in M(2,\C )$,  write $a=a(\gamma ), b=b(\gamma ), c=c(\gamma )$ and $d=d(\gamma )$ when it is necessary to stress the dependence of the entries on the matrix $\gamma$. It is known that     $\Psi (\gamma )=A_{\Psi(\gamma)}\sigma_{\Psi(\gamma)}$, where $A_{\Psi(\gamma)}$ is a linear  orthogonal map and $\sigma_{\Psi(\gamma)}$ is the reflection in the isometric sphere of $\Psi (\gamma )$.

Given $z_0\in \BQ^3$, let $P_{z_0}=z_0^*$ be the inverse point of $z_0$ with respect to $S_1^2(0)$. Choose $R_{z_0}>0$ such that $1+R_{z_0}^2=\|P_{z_0}\|^2$, let $\Sigma_{z_0}=S_{R_{z_0}}(P_{z_0})$ and let $\sigma_{z_0}$ be the reflection in $\Sigma_{z_0}$. Let $W_{z_0}=span_{\R}[j,z_0]$ be the plane spanned by $j$ and $z_0$, let $A_{z_0}$ be the reflection in $W_{z_0}$ and let $\gamma_{z_0}=A_{z_0}\circ \sigma_{z_0}$. It is easily seen that $\gamma_{z_0}$ is an orientation preserving isometry of $\BQ^3$, $o(\gamma_{z_0})=2$, $\gamma_{z_0}(z_0)=z_0$ and $\Sigma{\gamma_{z_0}}=\Sigma_{z_0}$. Hence $D_{\Gamma}=\gamma_{z_0}(D_{\gamma_{z_0}\Gamma\gamma_{z_0}})$ if and only if $D_{\Gamma}=D_{\Gamma}(z_0)$ if and only if $D_{\gamma_{z_0}\Gamma\gamma_{z_0}}=\gamma_{z_0}(D_{\Gamma})$.

We next recall some results proved in \cite{kijusiso} which will be needed in the sequel. Some are just partial statements of the complete results. The first result we state identifies the Poincaré bisectors in the ball model of hyperbolic $3$-space.

\begin{theorem}
\label{isosphereB3} 
Let $\gamma \in \SL (2,\C )$ with $ \gamma\notin \SU(2,\C )$.\\ 
Then $\Sigma_{\Psi (\gamma )}= \{u \in \BQ^3 \mid  \rho
(0,u)=$ $ \rho (u, \Psi (\gamma^{-1})(0))\}$, the
bisector of the geodesic segment linking $0$ and
$\Psi (\gamma^{-1})(0)$, is the isometric sphere of $\Psi (\gamma)$. Moreover $1+\frac{1}{|C|^2}=\vert P_{\Psi (\gamma )}\vert^2$ ,
$D_{\gamma}(0)= \BQ \cap
\mbox{Exterior}(\Sigma_{\Psi(\gamma )})$ and $P_{\Psi (\gamma)}^*=\Psi (\gamma^{-1})(0)$.
\end{theorem}

Using this result and the theory of hyperbolic spaces one gets explicit formulas for the Poincaré bisectors in the upper  half space model. In fact, we have the following results from \cite{kijusiso}.
  
\begin{proposition}
\label{isogammaupmodel}

Let $\gamma =
\left(
\begin{array}{ll}
a & b \\ c & d
\end{array}
\right) \in \SL(2,\C )$ and $\Sigma_{\gamma} = \eta_0^{-1}(\Sigma_{\Psi (\gamma)})$, where $\eta_0:\HQ^3\rightarrow \BQ^3$ is an isometry between the models (see \cite{elstrodt}). 
\begin{enumerate}
\item
$\Sigma_{\gamma} $ is an Euclidean sphere if and
only if $|a|^2+|c|^2 \neq 1$. In this case, its
center  and its radius are respectively given by
$P_{\gamma}=\frac{-(\overline{a}b+\overline{c}d)}{|a|^2+|c|^2-1}$,
$R^2_{\gamma}=\frac{1+\|P_{\gamma}\|^2}{|a|^2+|c|^2}$.
\item
$\Sigma_{\gamma}$ is a plane if and only if
$|a|^2+|c|^2 = 1$. In this case
$Re(\overline{v}z)+\frac{|v|^2}{2}=0, z\in \C$ is
a  defining equation of $\ \Sigma_{\gamma}$,  where
$v=\overline{a}b+\overline{c}d$.
\item $|\overline{a}b+\overline{c}d|^2= (|a|^2+|c|^2)(|b|^2+|d|^2)-1$
\item Suppose $c\neq 0$. Then $|\hat{P}_{\gamma}-P_{\gamma}|=\frac{|d-\overline{a}|}{|c|(|a|^2+|c|^2-1)}$. Moreover 
$\ \Iso_{\gamma}=\Sigma_{\gamma}$ if and only if
$d=\overline{a}$. In this
case we also have that $c=\lambda \overline{b}$,
with $\lambda \in \R$. If $c=0$ and $\infty \in \Sigma_{\gamma}$ then the same conclusion holds.
\item  Suppose that  $\ ISO_{\gamma}=$ $\Sigma_{\gamma}$ or that  $c=0$ and $\infty \in \Sigma_{\gamma}$.  Then 
$\tr (\gamma )\in \R$

\end{enumerate}
\end{proposition}

\begin{proposition}
\label{isogammaballmodel}
Let $\gamma = \left(
\begin{array}{ll}
a & b \\ c & d
\end{array}
\right)\in \SL (2,\C )$ and  $\Psi (\gamma ) =
\left(
\begin{array}{ll}
A & C^{\prime} \\ C & A^{\prime}
\end{array}
\right)$.
Then the following properties hold.
\begin{enumerate}
\item $\Iso_{\Psi (\gamma )}=\Sigma_{\Psi (\gamma )}$
\item $|A|^2=\frac{2+\|\gamma\|^2}{4}$, $|C|^2=\frac{\|\gamma\|^2-2}{4}$ and $|A|^2-|C|^2=1$
\item
$P_{\Psi (\gamma )}=
\frac{1}{-2+\|\gamma\|^2}\cdot [\
-2(\overline{a}b+\overline{c}d) +
[(|b|^2+|d|^2)-(|a|^2+|c|^2)]j\ ]$
\item
$\Psi (\gamma^{-1})(0)= P_{\Psi (\gamma )}^*=
\frac{1}{2+\|\gamma\|^2}\cdot [\
-2(\overline{a}b+\overline{c}d) +
[(|b|^2+|d|^2)-(|a|^2+|c|^2)]j\ ]$ (notation of 
inverse point w.r.t. $S^2$).
\item
$\|P_{\Psi (\gamma )}\|^2 = \frac{2+\|\gamma\|^2}{-2+\|\gamma\|^2}$
\item
$ R_{\Psi (\gamma )}^2= \frac{4}{-2+\|\gamma\|^2}$

\item $\Sigma_{\Psi (\gamma)} = \Sigma_{\Psi (\gamma_1)}$ if and only if $\gamma_1=\gamma_0\gamma$, $\gamma_0\in \SU(2,\C)$.
\end{enumerate}
\end{proposition}

 We will need to calculate the dihedral angle between two bisectors. For this we state the following result also obtained in \cite{kijusiso}.

\begin{lemma}
\label{angles}
Let $\gamma= \left(
\begin{array}{ll}
a & b \\ c & d
\end{array}
\right)\in  \SL (2, \C )$, $\pi_0 =\partial \HQ^n, \ n=2,3$, 
and $\theta$ the angle between $\Sigma_{\Psi
(\gamma )}$ and $\Sigma = \Sigma_{\Psi
(\gamma_1  )}$ with $\Sigma \cap \Sigma_{\Psi
(\gamma )} \neq \emptyset$. Then    $\cos (\theta ) = \frac{|1-\langle P_{\Psi (\gamma)}|  P_{\Psi
(\gamma_1)} \rangle |}{R_{\Psi (\gamma_1)}\cdot R_{\Psi
(\gamma)}}$.
\end{lemma}

Our next result proves that  \cite [Theorem 3.1]{kijusiso} holds in all dimensions.  With this result at hand we can now work in both models and carry over  information in a simple way.

\begin{theorem}
\label{geral}

Let $\gamma$ be an orientation preserving isometry of $\ \BQ^n$ and let $\Sigma_{\gamma}$be its isometric sphere. Then $\Sigma{\gamma}$ is the bisector of geodesic linking the origin $0$ and $\gamma^{-1}(0)$.

\end{theorem}

\begin{proof} Write $\gamma=A\sigma$, where $A$ is an orthogonal map and $\sigma$ the reflection in $\Sigma_{\gamma}$. Then $\gamma^{-1}(0)=\sigma A^{-1}(0)=\sigma (0)$ and hence $0$ and $\sigma (0)$ are inverse point with respect to $\Sigma_{\gamma}$. From this it follows that $\Sigma_{\gamma}$ is the bisector of the geodesic linking $0$ and $\sigma (0)$.\end{proof}

When looking for relations it is necessary to know the position of the bisectors relative to one another. In this direction it is easy to see that if  $\gamma,\gamma_1\in \PSL (2,\C)$ are non-unitary and  that $\gamma\gamma_1$ is also non-unitary  then  $\gamma_1^{-1}(\Sigma_{\gamma})\cap \Sigma_{\gamma\gamma_1} =\Sigma_{\gamma}\cap \Sigma_{\gamma_1}$. From the latter it follows that  $\gamma_1(\gamma_1(\Sigma_{\gamma})\cap \Sigma_{\gamma\gamma_1})= \Sigma_{\gamma_1^{-1}}\cap \Sigma_{\gamma}$. The Poincaré Theory tells us how to find relations.


\section{DF and DC Domains}
\label{dfdomains}

In this section we consider Kleinian groups which have a double Dirichlet domain or a Dirichlet-Ford domain. In particular, we settle a question on these groups raised in \cite{lakeland}.

Let $\Gamma$ be a Kleinian group and $\gamma \in \Gamma$. Write $\gamma = A_{\gamma}\sigma_{\gamma}$ where $\sigma_{\gamma}$ is the reflection in $\Sigma_{\gamma}$. Note that we have that $A_{\gamma}(\Sigma_{\gamma})=\Sigma_{\gamma^{-1}}$. Since $j$ and $\gamma^{-1}(j)$ are inverse points with respect to $\Sigma_{\gamma}$ we have that $A_{\gamma}(j)=j$.

\begin{lemma}
\label{Agamma}
Let $\gamma \in \Gamma$. Then the following hold.

\begin{enumerate}

\item $A_{\gamma} (j) =j$

\item  $A_{\gamma} (P_{\gamma})=\hat{P}_{\gamma}$

\item $A_{\gamma} (\infty )= \frac{b(\gamma)+\overline{c(\gamma)}}{d(\gamma)-\overline{a(\gamma)}}$, if $d(\gamma)\neq \overline{a(\gamma)}$

\item $A_{\gamma} (0)=\frac{b(|a|^2+|c|^2-)-\overline{c}(|b|^2+|d|^2-1)}{d(|a|^2+|c|^2-)+\overline{a}(|b|^2+|d|^2-1)}$

\item If $d(\gamma)=\overline{a(\gamma)}$ and $\gamma \notin \Gamma_j$ then $A_{\gamma} (\infty )=\infty$ and $A_{\gamma} (0 )=0$.

\end{enumerate}
\end{lemma}

\begin{proof} The first item, as seen above, is obvious. We have that $A_{\gamma}(P_{\gamma})= \gamma\circ \sigma_{\gamma}(P_{\gamma})=\gamma (\infty ) = \hat{P}_{\gamma}$. To prove the third item notice  that $A_{\gamma} (\infty )= \gamma (P_{\gamma})$. Using the expression of $P_{\gamma}$ and that  $\det (\gamma )=1$, the third item follows. To prove the fourth item, note that $\sigma_{\gamma}(P)=P_{\gamma}+\frac{R_{\gamma}^2}{\|P-P_{\gamma}\|}(P-P_{\gamma})$,   $\sigma_{\gamma}(0)=\frac{\|P_{\gamma}\|^2-R_{\gamma}^2}{\|P_{\gamma}\|^2}\cdot P_{\gamma}$ and hence $A_{\gamma}(0) =\gamma (\frac{\|P_{\gamma}\|^2-R_{\gamma}^2}{\|P_{\gamma}\|^2}\cdot P_{\gamma})$. From this the fourth  item follows easily. 

We now prove the last item. We have to consider all possible situations but apart from this the proof is straightforward. 

We first suppose that $d=d(\gamma)=0$. In this case $c=c(\gamma)\neq 0$. If $\gamma \not\in \Gamma_j$ then we have that $P_{\gamma}=\hat{P}_{\gamma}=0$ and $\sigma_{\gamma}(\infty)=0$. From this we have that $A_{\gamma}(\infty)=\gamma (0)=\infty$ and $A_{\gamma}(0)=\gamma (\infty)=0$.

If $d\neq 0$ and $c\neq 0$ then $P_{\gamma}=\hat{P}_{\gamma}$ and $\sigma_{\gamma}(P)=P_{\gamma}+\frac{R_{\gamma}^2}{\|P-P_{\gamma}\|}(P-P_{\gamma})$. From this we have that $A_{\gamma}(0)=\gamma (\frac{1-|d|^2}{c\overline{d}})=0$ and $A_{\gamma}(\infty)=\gamma (\hat{P}_{\gamma})=\infty$.

If $d\neq 0$ and $c= 0$ then $\Iso_{\gamma}$ does not exist, $|a|=1$ and hence $\Sigma_{\gamma}$ is a vertical plane. Using the defining equation of $\ \Sigma_{\gamma}$ we have that $\sigma_{\gamma}(0)=-(b\overline{a}+d\overline{c})=-bd$. Hence $A_{\gamma}(0)=\gamma (-bd)=0$ and $A_{\gamma}(\infty)=\gamma (\infty)=\infty$.\end{proof}

Observe that $A_{\gamma}=\eta_0^{-1}\circ A_{\Psi (\gamma)}\circ\eta_0$ and $\sigma_{\gamma}=\eta_0^{-1}\circ \sigma_{\Psi (\gamma)}\circ\eta_0$. Hence, working in $\BQ^3$, we see that $A_{\gamma^{-1}}=A_{\gamma}^{-1}$ and $\sigma_{\gamma^{-1}}=A_{\gamma}\circ \sigma_{\gamma}\circ A_{\gamma}^{-1}$.

\vp

\begin{theorem}
\label{dfdomain}

Let $\gamma$ be a Kleinian group. The following statements are equivalent.

\begin{enumerate}
\item There exist $z_0\neq z_1$ such that $D_{\Gamma}(z_0)=D_{\Gamma}(z_1)$.
\item $\Gamma$ has a Dirichlet-Ford domain.
\end{enumerate}

Moreover, if  $\ \Gamma$ is torsion free then it is not cocompact.

\end{theorem}

\begin{proof} Suppose first that $\Gamma$ has a Double Dirichlet domain and let $\Phi$ be the set of side-pairing transformations of $\ \Gamma$. We will work first in the ball model but, for simplicity, keep the notion of the upper half plane model. We may suppose that $z_1=0$ and $\Phi$ is taken with respect to $D_{\Gamma}$. Our hypothesis is that $D_{\Gamma}=D_{\Gamma}(z_0)$. Hence, given $\gamma \in \Phi$ there exists $\gamma_1\in \Gamma$ such that $\Sigma_{\gamma}=\Sigma_{\gamma_1}(z_0)$. In particular $z_0$ and $\gamma_1^{-1}(z_0)$ are inverse points with respect to $\Sigma_{\gamma}$ and thus $\gamma_1^{-1}(z_0)=\sigma_{\gamma}(z_0)$. From this we obtain that $\gamma(\gamma_1^{-1}(z_0))=\gamma\sigma(z_0)=A_{\gamma}(z_0)$. Consequently,  $\|\gamma(\gamma_1^{-1}(z_0))\|=\|A_{\gamma}(z_0)\|=\|z_0\|$ and hence, by \cite [Theorem IV.5.1]{ratcliffe}, we have that $0\in \Sigma_{\gamma_1\gamma^{-1}}(z_0)$. But $0$ belongs to the interior of $D_{\Gamma}(z_0)=D_{\Gamma}$; hence we have a contradiction unless $\gamma_1\gamma\in \Gamma_{z_0}$, i.e., $\Sigma_{\gamma_1\gamma}(z_0)$ does not exists. So we proved that $A_{\gamma}(z_0)=z_0$  for every $\gamma \in \Phi$. If $\lambda >0$, with $ \lambda z_0\in D_{\Gamma}$ and $\gamma\in \Phi$, we have that $\gamma^{-1} (\lambda z_0)=\sigma_{\gamma}(A_{\gamma}^{-1}(\lambda z_0))=\sigma_{\gamma}(\lambda z_0)$, i.e., $\lambda z_0$ and $\gamma^{-1} (\lambda z_0)$ are inverse points with respect to $\Sigma_{\gamma}$. Hence we proved that $\Sigma_{\gamma}=\Sigma_{\gamma}(\lambda z_0)$.   To complete this part of the proof, we now work in $\HQ^3$ and suppose that $z_0=j$. So we have that $A_{\gamma}(\lambda j)=\lambda j$, for all $\lambda >0$. By  continuity,  it follows that $A_{\gamma}(\infty)=\infty$. By Lemma~\ref{Agamma}, we have that $\Sigma_{\gamma}=\Iso_{\gamma}$. 

Now suppose that $\Gamma $ is torsion free. Let $V_{\gamma}=\BQ^3\cap\Sigma_{\gamma}\cap \R z_0$ and suppose that $V_{\gamma}=\{z_{\gamma}\}$. Then clearly $A_{\gamma}(z_{\gamma})=z_{\gamma}$ and hence $\gamma(z_{\gamma})=z_{\gamma}$. It follows that $o(\gamma)<\infty$ and $\Sigma_{\gamma}$ does not exist, a contradiction. So we have that $V_{\gamma}=\emptyset$ for all $\gamma\in \Phi$. From this it follows that $S^2\cap \R z_0$ is not covered by any $\Sigma_{\gamma}$ and hence $\Gamma$ is not cocompact.

We now prove the converse. Suppose that $\Gamma$ has a Dirichlet-Ford domain and let $\Phi$ be the set of side-pairing transformations of $\Gamma$. For every $\gamma \in \Phi$ we have that $d(\gamma)=\overline{a(\gamma )}$ (\cite{kijusiso} or Proposition~\ref{isogammaupmodel}) and hence, by Lemma~\ref{Agamma}, we have that $A_{\gamma}(0)=0$. In particular, $A_{\gamma}$ is an Euclidean linear isometry and hence $A_{\gamma}(\lambda j)=\lambda j$, for all $\lambda >0$. From this we have that for every $\lambda >0$, with $\lambda j$ an interior point of the domain,  $\gamma^{-1}(\lambda j)=\sigma_{\gamma}(\lambda j)$, i.e.,  $\Sigma_{\gamma}(\lambda j)=\Sigma_{\gamma}( j)=\Sigma_{\gamma}$. \end{proof}

Together with the results of Section IV of \cite{kijusiso}, this theorem gives a complete characterization of DF and DC domains and gives an algebraic criterium to decide whether a domain is a DF domain (and hence a DC domain).


\section{Bisector of Kleinian Groups}
\label{ford}

In this section we will frequently switch between the ball and upper half space models of hyperbolic 2 and 3-space and $\Gamma$ will stand for a discrete group of orientation preserving isometries. We keep notation as simples as possible to avoid  confusion.  Note that the results presented here for Kleinian groups are similar to those   in \cite{beardon} for Fuchsian groups.   A theory of pencils can also be developed in this case.

We first consider hyperbolic 3-space. We work in the ball model but, for simplicity, use the notation of the upper half plane model. Let $\gamma \in \Gamma$ and write $\gamma=A_{\gamma}\sigma_{\gamma}$. It follows that $\det (A_{\gamma})=-1$ and hence there exists $p_{\gamma}\in \BQ^3$, such that $A_{\gamma}(p_{\gamma})=-p_{\gamma}$. We have that $A_{\gamma}(P_{\gamma})=\gamma (\infty)=P_{\gamma^{-1}}$. Since $A_{\gamma}$ is a linear orthogonal map we obtain that $<p_{\gamma}|P_{\gamma}>=<-p_{\gamma}|P_{\gamma^{-1}}>$ and hence $p_{\gamma}$ is orthogonal to $P_{\gamma}+P_{\gamma^{-1}}$. In the same way we obtain that if $A_{\gamma}(w)=w$ then $w$ is orthogonal to $P_{\gamma}-P_{\gamma^{-1}}$. If $A_{\gamma}$ is diagonalizable then either $A_{\gamma}=-Id$ or it is the reflection in the plane $W: =<P|p_{\gamma}>=0$. In the first case it follows that $P_{\gamma^{-1}}=-P_{\gamma}$  and hence $\Sigma_{\gamma}$ and $\Sigma_{\gamma^{-1}}$ are disjoint. Consequently $\gamma$ is hyperbolic or loxodromic (see \cite{kijusiso}). In the second case it follows that $P_{\gamma}+P_{\gamma^{-1}}\in W$, $P_{\gamma}-P_{\gamma^{-1}}$ is orthogonal to $W$ and $\gamma$ is elliptic or parabolic if and only if $W\cap \Sigma_{\Psi (\gamma)}\neq \emptyset$. 

Given $\gamma\in \PSL (2,\C)$, choose $\gamma_0\in \SU (2,\C)$ such that $\gamma_1=\gamma_0^{-1}\gamma\gamma_0$ fixes $\infty$ when acting on $\HQ^3$. Then $c(\gamma_1)=0$, $\Sigma_{\gamma_1}=\Sigma_{\gamma\gamma_0}=\gamma_0^{-1}(\Sigma_{\gamma})$ and $(\Sigma_{\gamma},\Sigma_{\gamma^{-1}})=(\Sigma_{\gamma\gamma_0},\Sigma_{\gamma^{-1}\gamma_0})$$=\frac{|1-<P_{\gamma\gamma_0}|P_{\gamma^{-1}\gamma_0}>|}{R^2_{\gamma\gamma_0}}$ \\ $=\frac{1}{2(\|\gamma\|^2-2)}[-2|b|^2(Re(a\overline{d})-1)-(|a|^2+|d|^2)(\|\gamma\|-2)]$. 

In case $\gamma$ is hyperbolic then $a\overline{d}\in \R$ and so $(\Sigma_{\gamma},\Sigma_{\gamma^{-1}})=\frac{a^2+d^2}{2}=\frac{1}{2}\tr(\gamma^2)=-1+\frac{1}{2}[\tr(\gamma)]^2$.  In case $\gamma$ is elliptic or parabolic then $d=\overline{a}$ and $|a|=|d|=1$. Using this we get once more that $(\Sigma_{\gamma},\Sigma_{\gamma^{-1}})=-1+\frac{1}{2}[\tr(\gamma)]^2$. 

\begin{lemma}
\label{sigma2}

Let $\gamma\in \PSL (2,\C)$, be non unitary and non-loxodromic. Then we have that

\begin{enumerate}

\item $(\Sigma_{\gamma},\Sigma_{\gamma^{-1}})=-1+\frac{1}{2}[\tr(\gamma)]^2$

\item $\sigma_{\gamma^{-1}}\circ\sigma_{\gamma}=\gamma^2$

\item   $A_{\gamma}^2=Id$. In  particular, $A_{\gamma}$ is diagonalizable and it is either $-Id$ or a reflection.
\end{enumerate}

\end{lemma}

\begin{proof} The first item was proved above. To prove the second item we work in $\HQ^3$ and suppose that $\gamma$ fixes $\infty$ and $b(\gamma)=0$ if $\gamma$ is hyperbolic. We will  use the explicit formulas of $\ \Sigma_{\gamma}$ and $\sigma_{\gamma}$ freely and notation will follow that of the results of  Section~\ref{seciso}. In the parabolic case we have that $v=b$,  $\sigma_{\gamma}(P)= P-[<P|v>+\frac{|v|^2}{2}]\cdot \frac{v}{|v|^2}$ and $\sigma_{\gamma^{-1}}(P)= P-[<P|v>-\frac{|v|^2}{2}]\cdot \frac{v}{|v|^2}$. From this it follows that $\sigma_{\gamma^{-1}}\circ\sigma_{\gamma}(P)=P+2v=\gamma^2(P)$. In the elliptic case, set $a=a(\gamma)=e^{i\theta}$, $b=b(\gamma)\neq 0$. Then $\Sigma_{\gamma}$ and  $\Sigma_{\gamma^{-1}}$ are vertical planes with normal vectors $v_{\gamma}=e^{-i\theta}b$ and $v_{\gamma^{-1}}=-e^{i\theta}b$, respectively. Hence the angle between these two planes is $2\theta$, the angle of rotation of $\gamma$ around its axis, the intersection of the two planes. From this we get that $\sigma_{\gamma^{-1}}\circ\sigma_{\gamma}=\gamma^2$. In the hyperbolic case we have that $\gamma (P)=a(\gamma)^2P$, $\sigma_{\gamma}(P)=\frac{1}{a(\gamma)^2\overline{P}}$ and    $\sigma_{\gamma^{-1}}(P)=\frac{a(\gamma)^2}{\overline{P}}$. From this once again we get the desired formula.

To prove the last item recall that $\sigma_{\gamma^{-1}}=A_{\gamma}\sigma_{\gamma}A_{\gamma}$. The second item gives  that $\sigma_{\gamma^{-1}}\circ\sigma_{\gamma}=A_{\gamma}\sigma_{\gamma}A_{\gamma}\sigma_{\gamma}$ and hence $\sigma_{\gamma^{-1}}=A_{\gamma}\sigma_{\gamma}A_{\gamma}$. It follows that $A_{\gamma}=A_{\gamma}^{-1}$, finishing the proof. \end{proof}

\begin{lemma}
\label{composreflex}

Let $\Sigma_1$ and $\Sigma_2$ be distinct totally geodesic hyper surfaces of $\ \BQ^3$, $\sigma_i,i=1,2$ the reflection in $\Sigma_i$ and $\gamma=\sigma_1\circ\sigma_2$. Then $\gamma$ is non-loxodromic and  $(\Sigma_1,\Sigma_2)=\frac{1}{2}|\tr (\gamma)|$.
\end{lemma}

\begin{proof} We distinguish three cases: The surfaces are parallel, their intersection is a geodesic or they are disjoined. Working in $\HQ^3$, we may suppose that $\Sigma_2$ is the plane $x=x_0$ with $x_0>0$.

In the first case $\Sigma_1$ is a plane with equation $x=x_1$ with $x_1\neq x_0$. In this case $\sigma_1\circ\sigma_2$ is a parabolic element and the formula is easily seen to hold. In fact, the proof is along the same lines as that in the proof of the previous lemma.

In the second case $\Sigma_1$ is  a vertical plane making an angle of $\theta$ degrees with $\Sigma_2$. In this case $\sigma_1\circ\sigma_2$ corresponds to the rotation of  $2\theta$ degrees around $\Sigma_1\cap \Sigma_2$ and once again the formula holds trivially. Once again we refer to the proof of the previous lemma.

In the third case we may take $\Sigma_1=S_r(0)$ and $r<x_0$. The action on $\partial \HQ^3$ is given by $\sigma_1(x)=\frac{r^2}{x}$ and $\sigma_2(x)=-x+2x_0$. Since $0<r<x_0$ we find that $\gamma$ has exactly two fixed points in $\partial \HQ^3$, $x_2$ and $x_3$ say, and that the line segment $[x_2,x_3]$ is invariant under $\gamma$. From this it follows that the geodesic line  $l$, in $\HQ^3$,  linking $x_2$ and $x_3$ is the axis of $\gamma$. Mapping $l$ to a vertical line we may suppose that $c(\gamma)=b(\gamma)=0$. Hence we have that $a(\gamma)\cdot d(\gamma)=1$, $\gamma (\infty)=\infty$ and $\gamma(0)=0$. Note that now the $\Sigma_i$´s are Euclidean spheres $S_{R_i}(P_i)$. We have that $\infty =\gamma (\infty)=\sigma_1(\sigma_2(\infty))$ and hence $P_2=\sigma_2(\infty)=P_1=P$. It follows that $0=\gamma(0)=\sigma_1(\sigma_2(0))=[1+\frac{R_1^4}{R_2^4}]P$ and hence $P=0$. Consequently $\gamma(X)=\frac{R_1^2}{R_2^2}X$ and thus $a(\gamma)^2=\gamma (1)=\frac{R_1^2}{R_2^2}$. From this, and the definition of $(\Sigma_1,\Sigma_2)$,  the formula follows readily. \end{proof}

 \begin{lemma}
 \label{argum} Let $\gamma$ and $\gamma_1$ be such that $\Sigma_{\gamma}=\Iso_{\gamma}$ and $\Sigma_{\gamma_1}=\Iso_{\gamma_1}$. Then $\{\overline{c(\gamma)},i\overline{c(\gamma)},j\}$ is a basis of eigenvectors  of $A_{\gamma}$, $A_{\gamma}$ is the reflection in the plane $W_{\gamma}=span_{\R}[i\overline{c(\gamma)},j]$ and $\overline{c(\gamma)}$ is orthogonal to $W_{\gamma}$. Moreover, the angle between $W_{\gamma}$ and $W_{\gamma_1}$ is given by $\arg (\frac{c_1}{c_2})$.
 
 \end{lemma}
 
 \begin{proof}
Suppose that $\Sigma_{\gamma}=\Iso_{\gamma}$. Then we have that $P_{\gamma}=\hat{P}_{\gamma}$ and $A_{\gamma}(0)=0$. It follows that $A_{\gamma}$ is a linear isometry. We also have that $A_{\gamma}(P_{\gamma})=\gamma(\infty)=\hat{P}_{\gamma^{-1}}$. Since $A_{\gamma}$ is a linear orthogonal map reversing orientation we have that $A_{\gamma}(iP_{\gamma})=-i\hat{P}_{\gamma^{-1}}$. $A_{\gamma}$ fixes $\R^+j$ point wise. We have that $P_{\gamma^{-1}}=-\frac{\overline{a}^2}{|a|^2}P_{\gamma}$. Since $\tr (\gamma)\in \R$ we have that $\gamma$ is non-loxodromic and hence $A_{\gamma}^2=Id$. We have that $A_{\gamma}(aP_{\gamma})=\overline{a}P_{\gamma^{-1}}=-aP_{\gamma}$ and hence $A_{\gamma}(iaP_{\gamma})=iaP_{\gamma}$. Since $\overline{c}$ is an $\R$-multiple of $aP_{\gamma}$, it follows that $A_{\gamma}(\overline{c})=-\overline{c}$ an hence $\{\overline{c},i\overline{c},j\}$ is a basis of eigenvectors  of $A_{\gamma}$. So $A_{\gamma}$ is the reflection in the plane $W_{\gamma}=span_{\R}[i\overline{c},j]$ and $\overline{c}$ is orthogonal to $W_{\gamma}$. Hence   the angle between $W_{\gamma}$ and $W_{\gamma_1}$ is given by $\arg (\frac{c_1}{c_2})$. \end{proof}

It would be interesting to know if $\arg (\frac{c_1}{c_2})$ is also the dihedral angle between $\Sigma_{\gamma}$ and $\Sigma_{\gamma_1}$. This is true in the examples given in \cite{lakeland}. Note also that $\gamma$ is hyperbolic, elliptic or parabolic if either $\Sigma_{\gamma}\cap W_{\gamma}$ is empty, is a circle or consists of a single point. This follows also from the description of the relative position of $\ \Sigma_{\gamma}$ and  $\Sigma_{\gamma^{-1}}$ given in  \cite{kijusiso}. The lemma also suggests that a DF domain must be quite symmetrical. In the Fuchsian case the symmetry is with respect to the $i$-axes in $\HQ^2$ (see \cite{kijusiso, lakeland}).

In the Fuchsian case we take $\BQ^2$ as our model but, for simplicity of notation, use the notation of $\ \HQ^2$. Let $\gamma =\begin{pmatrix}
a & b\\
\overline{b} & \overline{a}
\end{pmatrix}$ with $|a|^2-|b|^2=1$ and $b\neq 0$. Writing $\gamma=A_{\gamma}\sigma_{\gamma}$ we have that $A_{\gamma}(P_{\gamma})=P_{\gamma^{-1}}$.  Here we have that $P_{\gamma}=\frac{-\overline{a}}{\overline{b}}$. Denote by $\{P_2,P_1\}$ and $\{P_4,P_3\}$, respectively the intersections $\Sigma_{\gamma}\cap S_1(0)$ and $\Sigma_{\gamma^{-1}}\cap S_1(0)$ (reading counterclockwise). The points $P_k,k=1,4$ can be obtained solving the equations $a\overline{b}z^2+2|b|^2z+\overline{a}b=0$ and $\overline{ab}z^2-2|b|^2z+ab=0$. We obtain that $P_1=\lambda P_{\gamma}$,  $P_4=\overline{\lambda}P_{\gamma^{-1}}$, where $\lambda =\frac{|b|}{|a|^2}(|b|+i)$, and, since $A_{\gamma}$ reverses orientation, $A_{\gamma}(P_1)=P_4$. Actually since $A_{\gamma}$ is an orthogonal map reversing  orientation and $A_{\gamma}(P_{\gamma})=P_{\gamma^{-1}}$, we have that $A_{\gamma}(iP_{\gamma})=-iP_{\gamma^{-1}}$. Noting that $P_{\gamma^{-1}}=-\frac{a^2}{|a|^2}P_{\gamma}$, and using the linearity of $A_{\gamma}$, we obtain that $A_{\gamma}^2=Id$. In fact, $P_{\gamma}+P_{\gamma^{-1}}$ or $P_{\gamma}-P_{\gamma^{-1}}$ is an eigenvector of $A_{\gamma}$. If necessary, the other eigenvector is obtained multiplying the one already obtained by $i$. Hence $A_{\gamma}$ is always diagonalizable and is a reflection. We summarize this in the following result.

\begin{lemma}
\label{fuchs}

Let $\gamma =\begin{pmatrix}
a & b\\
\overline{b} & \overline{a}
\end{pmatrix}$, with $|a|^2-|b|^2=1$ and $b\neq 0$, act on $\BQ^2$. Then $A_{\gamma}$ is   a reflection and  $P_{\gamma}+P_{\gamma^{-1}}$ or $P_{\gamma}-P_{\gamma^{-1}}$ is an eigenvector of $A_{\gamma}$.

\end{lemma}

If $\ \Sigma_1$ and $\Sigma_2$ are hyperbolic planes orthogonal to one another then the  product of the reflections in $\Sigma_1$ and $\Sigma_2$ is the reflection in their intersection.

\begin{lemma}
\label{lines}
Let $L_1$ and $L_2$ be two hyperbolic lines,  $\sigma_k$ the reflection in $L_k, k=1,2$ and $\gamma=\sigma_1\circ \sigma_2$. Then the following hold.

\begin{enumerate}

\item $\tr(\gamma)\in \R$ if and only if there exist a hyperbolic plane $\Sigma$ containing both $L_1$ and $L_2$.

\item Every hyperbolic or elliptic transformation is a product of two reflections in a line.

\item If $L_1\cup L_2\subset \Sigma$, where $\Sigma$ is a hyperbolic plane, then $\gamma$ is parabolic, elliptic or hyperbolic depending on the two lines being tangent, intersecting or disjoint. If such a $\Sigma$ does not exist then $\gamma$ is loxodromic.

\end{enumerate}

\end{lemma}

\begin{proof} We may suppose that $L_1$ is the $j$-axis and that $L_2$ is the line joining the points $z_0$ and $z_1$ in $\partial \HQ^3$. In this case we have that $\sigma_1=\begin{pmatrix}
i & o\\
0 & -i
\end{pmatrix}$ and $\sigma_2= \gamma_1^{-1}\sigma_1\gamma_1$, where $\gamma_1=\begin{pmatrix}
1 & -z_0\\
1 & -z_1
\end{pmatrix}$. Write $z_1=z_0+2Re^{i\theta}$, where $R$ is the radius of $L_2$. Then $\tr (\gamma)=-\frac{z_0e^{-i\theta}+R}{R}$. It follows that $\tr(\gamma)\in \R$ if and only if $z_0=te^{i\theta},t\in \R$ and hence $\frac{z_0}{z_1}\in \R$. Hence both lines are contained in a vertical plane and $|\tr(\gamma)|=2|1+\frac{t}{R}|$. From this it follows that $\gamma$ is parabolic if $t\in \{0,-2R\}$, elliptic if $-2R<t<0$ and hyperbolic if $t\notin [-2R,0]$. So from now on we may suppose that $\theta=0$, i.e., $z_0,z_1\in \R$.

In the hyperbolic case we may suppose $z_0=1$. The set of eigenvalues of $\gamma$  is $\{\frac{1-\sqrt{z_1}}{1+\sqrt{z_1}}, \frac{1+\sqrt{z_1}}{1-\sqrt{z_1}}\}$ and the fixed points are $\pm \sqrt{z_1}$. The image of the function $f(z_1)=\frac{1+\sqrt{z_1}}{1-\sqrt{z_1}}$ is $]0,1[$ and hence every hyperbolic element  is a product of two reflections. Note that $\gamma$ restricted to $L_2$ is an Euclidean isometry and hence $L_2\subset \Sigma_{\gamma}$. Note also that in this case the axis of $\gamma$, the hyperbolic line linking its fixed points, is orthogonal to $L_1$.

Finally, we consider the elliptic case.  In this case we have that $z_1>0$ and hence we may suppose $t=-1$.  We obtain that the fixed points of $\gamma$ are $\pm i\sqrt{z_1}$ and its spectrum is $\{\lambda_1,\lambda_2\}$ with $\lambda_2=\frac{2\sqrt{z_1}}{z_1+1}+i\frac{z_1-1}{z_1+1}$. The image of the function $f(z_1)=\frac{z_1-1}{z_1+1}$ is $]-1,1[$ and hence each elliptic element is the product of two reflections in a line. Note that the axis of $\gamma$ (the line linking the two fixed points) and the two lines, $L_1$ and $L_2$, all passes through the point $\sqrt{z_1}j$. \end{proof}

One can consider also the composition of the rotations in two lines. The situation is a bit more complicated but can be handled in a similar way.

Given $\gamma\in \PSL (2,\C)$ define the {\it canonical region} of $\gamma$ to be  \\ $\mbox{Canreg} (\gamma):=\{ P\in \HQ^3\ |\ \sinh [\frac{1}{2}\rho (P,\gamma (P))]<\frac{1}{2}|\tr(\gamma)| \}$ if $\ o(\gamma)\neq 2$ and $\mbox{Canreg} (\gamma):= Fix(\gamma)$ if $\ o(\gamma)=2$ (see \cite{beardon} for the Fuchsian case). Then clearly $\mbox{Canreg} (\gamma_0\gamma\gamma_0^{-1})=\gamma_0(\mbox{Canreg} (\gamma))$. From this it follows that $\gamma(\mbox{Canreg})= \mbox{Canreg} (\gamma)$.

In $\HQ^3$ we have that if $P=z+rj$ then  $\sinh [\frac{1}{2}\rho (P,\gamma (P))]= \frac{\|P-\gamma (P)\|}{2r}$. From this it follows easily that in the parabolic case, with $\gamma$ stabilizing $\infty$, $\mbox{Canreg} (\gamma)$ is the horoball $\{ P=z+rj\in \HQ^3\ |\ r>\frac{|b(\gamma)|}{2}\}$.

In the elliptic case we may suppose that $\gamma$ is a diagonal matrix and $a=a(\gamma)=e^{i\theta}$. In this case $\sinh [\frac{1}{2}\rho (P,\gamma (P))]=\frac{|1-a^2||z|}{2r}=\frac{|z|}{r}|\sin(\theta)|=\frac{|z|}{r}|\sin(\ln (a))|$. Let $L=\R j$. We have that $\cosh (\rho (P,\|P\|j))=\frac{\|P\|}{r}$ and hence $\sinh (\rho (P,\|P\|j))=\frac{\|z\|}{r}$. It follows that $\sinh [\frac{1}{2}\rho (P,\gamma (P))]=\sinh (\rho (P,L))|\sin (\ln a)|$.

We now look at the hyperbolic case.  Proceeding as in the elliptic case we obtain that \\ $\sinh [\frac{1}{2}\rho (P,\gamma (P))]=\frac{\|P\|}{r}|\sinh (\ln a)|$. In a similar way we obtain that $\sinh [\frac{1}{2}\rho (P,\gamma (P))]$ \\ $=\sinh (\rho (P,L))|\sinh (\ln a)|$. In this case $\frac{1}{2}|\tr(\gamma)| =\cosh (\ln a)$ and hence $P=$ $z+rj=$ $x+yi+rj\in \mbox{Canreg}(\gamma)$ if $x^2+y^2<\frac{r^2}{\sinh^2(\ln a)}$. Consider the line $l_{\gamma}$ given by the equations $y=0 $ and $r-x\sinh (\ln a)=0$ and also the  sphere, $\Sigma$ say, passing through $z$ and $\gamma(z)$ and orthogonal to $\partial (\HQ^3)$. Then $\Sigma$ is the sphere with center $\frac{(1+a^2)z}{2}$ and radius $\frac{|1-a^2||z|}{2}$. A simple calculation shows that $\Sigma$ is tangent to $\mbox{Canreg}(\gamma)$. In this case, note that if $\ \mbox{Canreg}(\gamma)=\mbox{Canreg}(\gamma_1)$ then $|\sinh (\ln a(\gamma))|=|\sinh (\ln a(\gamma_1))|$. From this it readily follows that $\gamma_1\in \{\gamma, \gamma^{-1}\}$.

In the elliptic case we obtain the cone $x^2+y^2<r^2co\tan^2(\theta_{\gamma})$, where $a(\gamma)=e^{i\theta}$. From this we also infer that if $\mbox{Canreg}(\gamma)=\mbox{Canreg}(\gamma_1)$ then $|\tan ( \theta_{\gamma})|=|\tan ( \theta_{\gamma_1})|$ and hence $\gamma_1\in \{\gamma, \gamma^{-1}\}$. In both cases $L$ is the axis of $\gamma$.

\vp
\noindent {\bf{Acknowledgment:}} The first  author is grateful to the Universidade Federal da Paraíba (UFPB-Brazil) and the Universidade Federal do Vale do São Francisco (UNIVASF-Brazil), for their hospitality while this research was being done.


\vspace{.2cm}

\noindent Instituto de Matem\'atica e Estatistica,\\ 
Universidade de S\~ao Paulo (IME-USP),\\ 
Caixa Postal 66281, S\~ao Paulo,\\ 
CEP  05315-970 - Brasil \\
email: ostanley@usp.br

\vspace{.2cm}

 \noindent Universidade Federal do Vale do São Francisco,\\
  Colegiado de Engenharia Mecânica,\\ 
Avenida Antonio Carlos Magalhães, 510, \\
Colegiado de Engenharia Mecanica,\\
Santo Antônio\\
48902-300 - Juazeiro, BA - Brasil.\\
e-mail:  cirino.lima@univasf.edu.br

\vspace{.2cm}

\noindent Departamento de Matematica\\
Universidade Federal da Paraiba\\
João Pessoa - PB - Brasil\\
e-mail: andrade@mat.ufpb.br

\vspace{.25cm}


\begin{thebibliography}{99}
\itemsep=-2pt

 
\bibitem{beardon} A. F. Beardon, The Geometry of Discrete Groups, Springer Verlag NY, 1983.
  
\bibitem{bridson} M. R. Bridson, A. Haefliger,  Metric Spaces of Non-Positive Curvature, Springer Verlag, Berlin, 1999.
 
\bibitem{elstrodt} J, Elstrodt, F. Grunewald, J. Mennicke, Groups 
Acting on Hyperbolic Space, Springer Verlag,
Berlin Heidelberg, 1998.


\bibitem{gromov} M. Gromov, Hyperbolic Groups, in Essays in Group Theory, M. S. R. I. Publ. 8, Springer Verlag, 1987, 75-263.


 
\bibitem{kijusiso} E. Jespers, S.O. Juriaans, A. Kiefer, A. De A. E Silva, A.C. Souza Filho,  Poincaré Bisectors in Hyperbolic Spaces, Submitted.


 
\bibitem{johansson} S. Johansson, On fundamental domains of arithmetic Fuchsian Groups, Math. Comp. 69 (2000),no.229, 339-349.


\bibitem{katok} S. Katok, Reduction Theory for Fuchsian Groups,  Math. Ann. 273, 461470 (1986).


\bibitem{lakeland}  G. S. Lakeland,  Dirichlet-Ford Domains and Arithmetic Reflection Groups, Pacific J. Math., to appear.


\bibitem{ratcliffe} J.G., Ratcliffe, Foundations of Hyperbolic Manifolds, Springer Verlag, New York, 1994.



\bibitem{swan} R. Swan, Generators and relations for certain special linear groups, Adv. in Math., v. 6, 1971, pp.
1-77.

 


\end{thebibliography}
\end{document}